\documentclass[12pt,reqno]{amsart}
\usepackage{amsaddr}
\usepackage[ocgcolorlinks,hyperfootnotes=false,colorlinks=true,citecolor=blue,linkcolor=blue,urlcolor=blue]{hyperref}
\author{Tirthankar Bhattacharyya and Ritul Duhan}
\address{Department of Mathematics, 	Indian Institute of Science, \\
	Bangalore 560012, India}
\email{tirtha@iisc.ac.in; ritul2023@iisc.ac.in}

\usepackage[capitalize,nameinlink]{cleveref}

\usepackage{amsmath, color, bm, amscd, tikz-cd}
\usepackage{csquotes}

\usepackage{tikz}

\newtheorem{thm}{Theorem}[section]

\newtheorem{cor}[thm]{Corollary}
\newtheorem{lem}[thm]{Lemma}

\newtheorem{notation}[thm]{Notation}

\newtheorem{theorem}{Theorem}
\newtheorem{defn}[thm]{Definition}

\numberwithin{equation}{section}

\def\tilde{\widetilde}

\def\cS{\mathcal{S}}

\def\bar{\overline}

\def\C{\mathbb{C}}

\def\epsilon{\varepsilon}

\def\min{\mathrm{min}}

\def\phi{\varphi}

\def\cS{\mathcal{S}}
\def\T{\mathbb{T}}
\def\cT{\mathcal{T}}

\def\Z{\mathbb{Z}}
\def\ucp{{\it{ucp }}}
\def\cp{{\it{cp }}}

\newcommand{\zbar}{\overline z}

\newcommand{\Bmn}{\mathcal{B}_m^{(n)}}
\newcommand{\Ym}{\mathcal{Y}_m}

\begin{document}
	
	\title[Pure matrix states on block Toeplitz matrices]{Pure matrix states on block Toeplitz matrices}

	\begin{abstract}
		Let $\mathcal{T}_{n,m} = \mathcal T_n(M_m(\mathbb{C}))$ denote the operator system of all block Toeplitz matrices 
		$ T = (( T_{i-j}))_{i,j=1}^n$ with entries $T_k \in M_m(\mathbb C) $
		% T_{k} = \left[ t^{(k)}_{p-q} \right]_{p,q=1}^m.
               \[
		T
		=
		\begin{pmatrix}
			T_0      & T_{-1} & \cdots & T_{-(n-1)} \\
			T_1      & T_0    & \cdots & T_{-(n-2)} \\
			\vdots            & \vdots          & \ddots & \vdots \\
			T_{n-1}  & T_{n-2}& \cdots & T_0
		\end{pmatrix}
		\in M_{mn}(\mathbb{C}).
		\]
		 We characterize all pure unital completely positive ({\it{ucp}}) maps from $\mathcal{T}_{n,m}$ to $M_m(\mathbb{C})$. 

Working through the Stinespring isometry $V = (V_1, \ldots, V_n)^t \colon \mathbb{C}^m \to \mathbb{C}^{mn}$
and the matrix-valued polynomial $Q_V(z) = \sum_{i=1}^n z^{n-i} V_i$, we prove that
$\varphi$ is pure if and only if it admits a unique pure \ucp extension to
$M_{mn}(\mathbb{C})$ if and only if $Q_V$ has degree $n-1$ with all its roots on
the unit circle $\T$.

Every such pure $\varphi$ induces a \ucp map $\Phi_{Q_V}$ on $C(\mathbb{T}, M_m(\mathbb{C}))$ given by
\[
  \Phi_{Q_V}(f) = \int_{\mathbb{T}} Q_V(z)^* f(z) Q_V(z)\, dz.
\]
Let $\mathcal{Y}_m$ be the compact convex set of all \ucp maps from $C(\mathbb{T}, M_m(\mathbb{C}))$ to
$M_m(\mathbb{C})$. Endowing $\mathcal{Y}_m$ with the matricial Monge-Kantorovich metric $\rho$, via an analysis of the extreme points of $\mathcal{Y}_m$ together with a point-splitting lemma, we show that the induced maps as above are
$\rho$-dense in $\mathcal{Y}_m$. Consequently, if $\Bmn$ denotes
the set of normalized $\Phi_{Q_V}$ where $Q_V$ is as above, then the Hausdorff distance 
$d_H(\Bmn, \mathcal{Y}_m) \to 0$ as $n \to \infty$, extending known results of approximation of positive regular Borel measures on the unit circle to the setting of matrix-valued  completely positive maps.
\end{abstract}
	\maketitle
	
	{\footnotesize \noindent 2020 Mathematics Subject Classification: 46L07, 47L25, 58B34, 81R15. \\
		Keywords: Block Toeplitz operator systems, \ucp maps, matrix-valued trigonometric polynomials, unique \cp extension, extreme points, Monge-Kantorovich metric, Hausdorff and Gromov-Hausdorff convergence.}
	
	\smallskip
\noindent\textbf{Funding:} \\
{\footnotesize{T. Bhattacharyya is supported by the J C Bose Fellowship  JCB/2021/000041 of SERB and R. Duhan is supported by the Prime Minister's Research Fellowship (PMRF ID 0202985) of the Government of India.  This research is supported by the DST FIST program-2021 [TPN-700661].}}

 \smallskip

\noindent {\footnotesize{\textbf{Acknowledgement:} \\
The authors are thankful to B. V. Rajarama Bhat and Abhay Jindal for useful discussions. AI tools have been used to improve writing.}}
	
	\section{Introduction}\label{sec:intro}
	
\subsection{The main contribution}	The characterization of pure states and extremal channels constitutes a foundational theme in operator algebras and quantum information theory. 

By the classical Stinespring dilation theorem, a \ucp map  on a $C^*$-algebra is pure if and only if its minimal Stinespring dilation yields an irreducible representation. This is no longer available once one passes from $C^*$-algebras to general {\em operator systems} - subspaces of $\mathcal B(\mathcal H)$ that contain the identity and are closed under adjoints. Characterizing pure \ucp maps on an operator system is considerably more delicate, primarily because the extension of a \ucp map from an operator system to its generated $C^*$-algebra is rarely unique, see ~\cite{Clou} for intricate conditions ensuring unique extension. 

The Toeplitz matrices have been studied from time immemorial. However, the attention given to them from the points of view of spectral truncation (see ~\cite{CvS}) or operator system theory (see ~\cite{Far_TAMS}) seems to be recent.  A novel idea introduced in ~\cite{CvS} relates purity of a state on the Toeplitz operator system to the location of roots of an associated polynomial. This characterization of pure states inspired ~\cite{DJ} to give an explicit characterization of pure \ucp maps defined on the (scalar) Toeplitz operator system $\mathcal T_d= \mathcal T_d(\mathbb C)$ and taking values in $M_m(\mathbb C)$, for $d,m$ independent of one another, in terms of the associated positive $M_m(\mathbb C)$-valued trigonometric polynomial: purity corresponds to a \emph{rank-deficiency} condition, see [~\cite{DJ}, Theorem 1]. 
	
	Block Toeplitz matrices arise naturally as compressions of multiplication operators governed by matrix-valued symbols, and this makes $\mathcal T_{n,m}$ a bridge between operator theory, harmonic analysis, and matrix-convex geometry. Let $C(\T, M_m(\C))$ be the space of all matrix-valued continuous functions on the unit circle. Considering $C(\T, M_m(\C))$ as a subalgebra of $\mathcal B(L^2(\T), \mathbb C^m)$ by identifying $f \in  C(\T, M_m(\C))$ with the multiplication operator $M_f$ on $\mathcal B(L^2(\T), \mathbb C^m)$ and letting $P_n$ to be the projection onto $\mathcal H_n =  \mathbb C^m \oplus \cdots \oplus \mathbb C^m $ ($n$ times),  the compression $P_nM_fP_n$ acting on $\mathcal H_n$ is represented by the block Toeplitz matrix
	\[
	\begin{pmatrix}
		\hat{f}(0)      & \hat{f}(-1) & \cdots & \hat{f}(-n+1) \\
		\hat{f}(1)      & \hat{f}(0)   & \cdots & \hat{f}(-n+2) \\
		\vdots            & \vdots          & \ddots & \vdots \\
		\hat{f}(n-1)  & \hat{f}(n-2) & \cdots & \hat{f}(0)
	\end{pmatrix}.
\]
Thus the finite block Toeplitz operator system $\cT_{n,m}$ may be realised as $\{P_nM_fP_n: f \in C(\T, M_m(\C)) \}.$ 

This raises the hope of following the well-charted path of associating a positive $M_m(\mathbb C)$-valued trigonometric polynomial with a \ucp map. This naive attempt quickly runs into difficulty. Every \ucp map $\varphi \colon \mathcal T_{n,m} \to M_m(\mathbb C)$ does give rise to an $M_m(\mathbb C)$-valued positive-semidefinite trigonometric polynomial $P_\varphi$ by considering $\eta(z) = (I, \zbar I, \zbar^2 I, \ldots , \zbar^{n-1} I)^t$, where the identity matrices are $m \times m$, letting $M(z) = \eta(z) \eta(z)^*$ and defining $\mathcal P_\varphi$  to be $\varphi (M(z))$. While complete positivity of $\varphi$ implies that $\mathcal P_\varphi \ge 0$ on the unit circle, $\varphi$ and $\mathcal P_\varphi$ are not related by
	$$ \varphi (P_nM_fP_n) = \int_{\mathbb T} f(z) \mathcal P_\varphi(z)\; dz$$
	where $dz$ is the normalized arc length measure on $\mathbb T$, thereby rendering the powerful methods of applying the classical Fej\'er-Riesz Theorem invalid in  the block case. The block Toeplitz operator system also differs in the fact that not all positive maps are automatically completely positive.

There is a well-studied concept of unique extension property in noncommutative Choquet theory in the seminal works ~\cite{Arveson72, Arveson08} and established in full generality in ~\cite{DK15, DK25}. The noncommutative Choquet boundary is intrinsically governed by UEP (unique extension property) - the phenomenon where a \ucp map on an operator system admits a unique \ucp extension to its generated C$^*$-algebra. If this extension is an irreducible $*$-representation, it is called a boundary representation. These boundary representations are also studied in ~\cite{DM, Kleski} and have vast applications in calculating C$^*$-envelope, to check the hyperigidity of an operator system, and in finding maximal dilations in operator theory. So, it is worthwhile to consider the notion of unique \ucp extension.  
	
	The contribution of this note lies in proving that purity forces unique extension as well as a strictly stronger degeneracy - the vanishing of an entire matrix, not merely a proper subspace of it.
	
	\begin{theorem} \label{TheMainTheorem} 
Let $\varphi \colon \mathcal T_{n,m} \to M_m(\mathbb C)$ be a \ucp map. The following are equivalent
		
		(i) $\varphi $ is pure.
		
		(ii) $\varphi $ has a unique pure \ucp extension to $ M_{mn}(\mathbb C)$. 
		
		(iii)  $\varphi(T) = V^* T V$ for an isometry $V = (V_1, V_2, \ldots , V_n)^t$ from $\mathbb C^m$ into $\mathbb C^{mn}$ which satisfies the following condition: the $M_m(\mathbb C)$-valued polynomial $Q_V(z) = \sum\limits_{i=1}^n z^{n-i}V_i $ is of degree $n-1$ has all its roots (i.e., complex points where $Q_V$ reduces to the zero polynomial) on the unit circle $\mathbb T$.
	\end{theorem}
Our proofs follow a detailed analysis of the polynomial $Q_V$. 

\subsection{Hausdorff convergence}
\begin{defn}
A polynomial $Q_V$ where $V$ is an isometry as above will be called a regular polynomial if its roots lie on the unit circle $\mathbb{T}$.
\end{defn}
Regular polynomials are going to appear frequently in this paper. The analogue of Theorem ~\ref{TheMainTheorem} for states - which is the motivation for obtaining Theorem ~\ref{TheMainTheorem} - was used  in  ~\cite{Hek} to obtain that any state on $C(\T)$ can be approximated in the Monge--Kantorovich metric by a pure state. 

Approximation of a probability measure $\mu$ on the unit circle is a classical theme. The Szego quadrature produces, for every $n$, a measure $\mu_n$ which is moment-exact with $\mu$ on trigonometric polynomials of degree up to $n-1$, converges weak-* to $\mu$ and has nodes on the unit circle which are zeroes of a paraorthogonal polynomial, see section 7 of \cite{JNT}. The next results are in this vein.

We first write $\varphi$ explicitly in terms of $Q_V$. Let $C^*(\mathbb Z , M_m(\mathbb C))_{(n)}$ be the matrix-valued Fej\'er-Riesz operator system - in the language of ~\cite{CvS} - of all trigonometric polynomials $A(z) = \sum_{k=-n+1}^{n-1} A_k z^k$ where the $A_k$ are from $M_m(\mathbb C)$. With $T$,  associate the element $F_T(z) = \sum_{k=-(n-1)}^{n-1} z^{-k} T_k$ of $C^*(\mathbb Z , M_m(\mathbb C))_{(n)}$. By Fourier extraction, we have
		\[
		T_{i-j} = \int_{\mathbb{T}} z^{i-j} F_T(z)\; dz
		\]
		where $dz$ is the normalized arc length measure on $\T$. So,
		\[
		V^* T V = \sum_{i,j=1}^n V_i^* T_{i-j} V_j = \sum_{i,j=1}^n V_i^* \left( \int_{\mathbb{T}} z^{i-j} F_T(z) \;dz \right) V_j
		\]
		which gives that 
\begin{equation}\label{phiFandintegral}
		\varphi(T) = \int_{\mathbb{T}} Q_V(z)^* F_T(z) Q_V(z)\; dz.
		\end{equation}
Thus, starting with a pure \ucp map $\varphi \colon \mathcal T_{n,m} \to M_m(\mathbb C)$, we are naturally led to define 
\begin{equation} \label{ucponwholeC}\Phi_{Q_V}: C(\mathbb{T}, M_m(\mathbb{C})) \to M_m(\mathbb{C}) \text{ by }
		\Phi_{Q_V}(f) =\int_{\mathbb{T}} Q_V(z)^* f(z) Q_V(z) \, dz, \end{equation}
where $Q_V$ is a regular polynomial. This leads us to pose the question of how large the set of \ucp maps induced by regular polynomials is in 
 {\em the convex set, $\mathcal{Y}_m$ of all unital completely positive (\it{ucp}) linear maps $\Phi: C(\mathbb{T}, M_m(\mathbb{C})) \to M_m(\mathbb{C})$.}
 
   The relevant topology is the bounded weak (BW) topology on $\mathcal{Y}_m$. A net $\phi_\alpha$ converges to $\phi$ in $\mathcal{Y}_m$ in the BW topology if and only if
	\[
	\|\phi_\alpha(f) - \phi(f)\|_{M_m(\mathbb{C})} \longrightarrow 0 \quad \forall f \in C(\mathbb{T}, M_m(\mathbb{C})).
	\]

Using [~\cite{Paul}, Theorem 7.4] we get that $\mathcal{Y}_m$ is a compact convex subset of the bounded linear operators $\mathcal{B}(C(\mathbb{T}, M_m(\mathbb{C})), M_m(\mathbb{C}))$ under the BW topology. It is also metrizable.

\begin{notation} \label{L} Denote $\mathcal{L} = \{ f \in C^1(\mathbb{T}, M_m(\mathbb{C})) : \, \|f\|_\infty \le 1\; \text{and}\; \|f'\|_\infty \le 1 \}$ 	where $\|f\|_\infty = \sup_{z \in \mathbb{T}} \|f(z)\|_{M_m(\mathbb{C})}$.
\end{notation}

\begin{defn} \label{rho}
	Let $C^1(\mathbb{T}, M_m(\mathbb{C}))$ denote the space of continuously differentiable matrix-valued functions on $\mathbb{T}$. The matricial Monge--Kantorovich metric $\rho$ (see ~\cite{Kerr}) on $\mathcal{Y}_m$ is defined as
	\[
	\rho(\phi, \psi) = \sup \left\{ \|\phi(f) - \psi(f)\|_{M_m(\mathbb{C})} : f \in \mathcal L \right\}.
	\]
\end{defn}
The metric $\rho$ induces the bounded weak (BW) topology on $\mathcal{Y}_m$. 
In the case when $m=1$, i.e., when $\varphi$ is a state, $V$ is a unit vector $v$ and commutativity implies that $\varphi$ is implemented by an absolutely continuous measure whose Radon-Nikodym derivative with respect to $dz$ is $|Q_v|^2$. It is then possible to multiply the  Radon-Nikodym derivatives and that is a crucial step in showing that any state on $C(\T)$ can be approximated by pure states (of the truncated system), see [~\cite{Hek}, Section 3.1]. Following altogether a different route - finding the extreme points of $\mathcal{Y}_m$, proving a point-splitting lemma and analyzing the polynomials - we prove the following approximation.
\begin{theorem}\label{thm:existenceQ_V}
	Given an arbitrary $\Phi \in \mathcal{Y}_m$ and any $\epsilon > 0$, there exists a regular matrix-valued polynomial $Q_V(z) \in M_m(\C[z])$ such that
	the induced \ucp map ~\eqref{ucponwholeC} is pure and satisfies $\rho(\Phi, \Phi_{Q_V}) < \epsilon$.
\end{theorem}

A different way of asserting that for a large $n$, the convex cone of all completely positive maps from $\mathcal T_{n,m}$ to $M_m(\mathbb C)$ is nearly the whole of $\Ym$ is as follows.

\begin{defn}
	For each $n \in \mathbb{N}$, define 
	\[
	\Bmn = \left\{ \Phi_{Q_V} \in \mathcal{Y}_m :  Q_V \text{ is regular, } \text{deg}(Q_V) = n - 1 \text{ and } \int_\T Q_V(z)^*Q_V(z)\;dz=1\right\}.
	\]
\end{defn}

\begin{theorem}\label{thm:hauscong}
	 The Hausdorff distance between $\Bmn$ and $\Ym$ converges to zero as the polynomial degree $n$ approaches infinity: 
	\[
	\lim_{n \to \infty} d_H(\Bmn, \Ym) = 0.
	\]
\end{theorem}
	
 Although we shall not use it in this note, for the sake of beauty we mention that counting dimensions and using the methods of ~\cite{Far_TAMS}, it is straightforward that $A(z) \rightarrow L_A$ given by $L_A(T) = $ Trace ($\sum_{k=-n+1}^{n-1} A_k T_{-k})$ for $T = ((T_{i-j})) \in \cT_{n,m} $ establishes a complete order isomorphism between the dual $\cT_{n,m}^d$ (for dual of an operator system see ~\cite{CE}) and $C^*(\mathbb Z , M_m(\mathbb C))_{(n)}$.

	\section{Positive semidefinite Block Toeplitz matrices}

This section gives a quick description of positive elements of $\mathcal{T}_{n,m}$, a different one from the characterization in ~\cite{Far_TAMS}.	Since the block Toeplitz matrices contain the $mn \times mn$ Toeplitz matrices and the $C^*$- algebra generated by $mn \times mn$ Toeplitz matrices is $M_{mn}(\C)$, the $C^*$- algebra generated by $\cT_{n,m}$ is the algebra of matrices $M_{mn}(\C)$.
	
	In the following, we decode the structure of a positive semidefinite element $T$ of $\cT_{n,m}$. Suppose $\operatorname{rank}(T)=r$. Then there exists
	$Y \in M_{mn,r}(\mathbb{C})$ such that $ T =YY^*.$
	Write
	\[
	Y
	=
	\begin{pmatrix}
		Y_0^* & \cdots & Y_{n-1}^*
	\end{pmatrix}^*,
	\qquad
	Y_j \in M_{m,r}(\mathbb{C}), \; j = 0,\ldots,n-1.
	\]
	Define the upper and the lower submatrices
	\[
	Y_U
	=
	\begin{pmatrix}
		Y_0^* & \cdots & Y_{n-2}^*
	\end{pmatrix}^*
	\text{	and } 
	Y_L
	=
	\begin{pmatrix}
		Y_1^* & \cdots & Y_{n-1}^*
	\end{pmatrix}^*.
	\]
	By the block Toeplitz structure of $T$, we have $ Y_U Y_U^* = Y_L Y_L^*$. 
	Thus there exists a unitary matrix $U \in M_r(\mathbb{C})$ such that $ 	Y_L = Y_U U$. It follows that
	\[
	Y
	=
	\begin{pmatrix}
		Y_0^* &
		(Y_0 U)^* &
		\cdots &
		(Y_0 U^{\,n-1})^*
	\end{pmatrix}^*.
	\tag{1}
	\]
	Hence, for every positive semidefinite block Toeplitz matrix $T$, there exists $Y_0 \in M_{m,r}(\C)$ and a unitary $U \in M_r(\C)$ such that 
	$ 	T= ABB^*A^*,$
	where
	\[
	A
	=
	\begin{pmatrix}
		Y_0 & 0 & \cdots & 0 \\
		0 & Y_0 & \cdots & 0 \\
		\vdots & \vdots & \ddots & \vdots \\
		0 & 0 & \cdots & Y_0
	\end{pmatrix}
	\in M_{mn,rn}(\mathbb{C})
	\text{ 	and } 
	B
	=
	\begin{pmatrix}
		I \\
		U \\
		\vdots \\
		U^{n-1}
	\end{pmatrix}
	\in M_{rn,r}(\mathbb{C}).
	\]
	So, $T$ is of the form
	
	\[
	T
	=
	\begin{pmatrix}
		Y_0 &        &        \\
		& \ddots &        \\
		&        & Y_0
	\end{pmatrix}
	\begin{pmatrix}
		I \\
		U \\
		\vdots \\
		U^{\,n-1}
	\end{pmatrix}
	\begin{pmatrix}
		I & U^* & \cdots & (U^{\,n-1})^*
	\end{pmatrix}
	\begin{pmatrix}
		Y_0^* &        &        \\
		& \ddots &        \\
		&        & Y_0^*
	\end{pmatrix}.
	\]
	Getting hold of a unitary $U_0$ such that $U=U_0 D U_0^*$ where $D = \text{diag}(\lambda_1, \lambda_2, \ldots, \lambda_r)$ with each $\lambda_i \in \T$ and letting $Y = Y_0U_0 \in M_{m,r}(\C)$, we have
	
	\[
	T
	=
	\begin{pmatrix}
		Y &        &        \\
		& \ddots &        \\
		&        & Y
	\end{pmatrix}_{mn\times rn}
	\begin{pmatrix}
		I \\
		D \\
		\vdots \\
		D^{\,n-1}
	\end{pmatrix}_{rn \times r}
	\begin{pmatrix}
		I & D^* & \cdots & (D^{\,n-1})^*
	\end{pmatrix}_{r \times rn}
	\begin{pmatrix}
		Y^* &        &        \\
		& \ddots &        \\
		&        & Y^*
	\end{pmatrix}_{rn \times mn}.
	\]
	
	The reduction to one variable case is noteworthy. Let $T$ be an $n \times n$ positive semidefinite  Toeplitz matrix of rank$(T) = r$. Then the above form of $T$ can be reduced to the given Toeplitz matrix by taking $m=1$. So letting $ Y =(y_1, y_2, \ldots, y_r)$, we introduce the $n \times r $ Vandermonde matrix \[W= \begin{pmatrix}
		1 & 1 & 1 & \cdots & 1 \\
		\lambda_1 & \lambda_2 & \lambda_3 & \cdots & \lambda_r \\
		\lambda_1^2 & \lambda_2^2 & \lambda_3^2 & \cdots & \lambda_r^2 \\
		\vdots & \vdots & \vdots & \ddots & \vdots \\
		\lambda_1^{n-1} & \lambda_2^{n-1} & \lambda_3^{n-1} & \cdots & \lambda_r^{n-1}
	\end{pmatrix} \in M_{n,r}(\C).\]
	
	Then by computations it can be easily verified that $T$ decomposes as \[ T=W \begin{pmatrix}
		|y_1|^2 &        &        \\
		& \ddots &        \\
		&        & |y_r|^2
	\end{pmatrix} W^*.\]

	\section{Pure ucp maps} \label{PureUCP}
	Given an operator system $\mathcal{S}$, let $\mathrm{CP}_\mathcal{K}(\mathcal{S})$ denote the cone of completely positive ({\it{cp}}) maps from $\mathcal{S}$ to $\mathcal{B}(\mathcal{K})$, and let $\mathrm{UCP}_\mathcal{K}(\mathcal{S})$ be the convex subset of \cp maps that are unital. A map $\phi \in \mathrm{UCP}_\mathcal{K}(\mathcal{S})$ is \emph{pure} if whenever $\phi - \psi$ is \cp (we write
	$\phi \succeq \psi$ in this case) for some $\psi \in
	\mathrm{CP}_\mathcal{K}(\mathcal{S})$, then there exists
	$0 \le t \le 1$ such that \[
	\psi = t \phi.\]
	If $\mathcal{K}$ is one-dimensional, a pure \ucp map from $\mathcal{S}$ to
	$\mathcal{B}(\mathcal{K})$ is just a pure state and when $\mathcal{K}$ is
	finite-dimensional, the elements of $\mathrm{UCP}_\mathcal{K}(\mathcal{S})$ are called \emph{matrix states}.
	
	\medskip
	
	A pure \ucp map cannot be written as a non-trivial convex combination of other \ucp maps, i.e., if $\phi$ is a pure \ucp map and $\phi = \lambda\phi_1 + (1-\lambda)\phi_2$ for some $0<\lambda <1$ then $\phi = \phi_1 = \phi_2$. Hence, $\phi$ is an extreme point of the convex set $\mathrm{UCP}_\mathcal{K}(\mathcal{S})$. A map $\phi$ from $\mathcal{S}$ to $\mathbb C$ is a pure state if and only if $\phi$ is an extreme point of the state space.
	
	\smallskip 
	
	In this paper, we shall consider $\cS$ to be space of all block Toeplitz matrices $\cT_{n,m}$ and $\mathcal{K}$ to be of dimension $m$.

	\begin{lem}\label{lem:pureV}
		Let $\phi$ be a pure \ucp map from $\cT_{n,m}$ to $M_{m}(\C)$. Then there exists an isometry $V \in M_{mn, m}(\C)$ such that $\phi(A) = V^*AV, \; \forall \;  A\in \cT_{n,m}$. 
	\end{lem}
	
	\begin{proof}
		 This can be proved as a straightforward application of Theorem B of \cite{Far_JLMS} which extends $\phi$ to the whole of $M_{mn}(\C)$ as a pure map. And that all pure \ucp maps from $M_{mn}(\C)$ to $M_{m}(\C)$ are of the form $V^*AV$ is a folklore.
	\end{proof}

 By applying the representation established in Lemma \ref{lem:pureV}, we now evaluate the action of pure \ucp map $\phi$ on a positive block Toeplitz matrix $T \in \cT_{n,m}$.
	
	\smallskip
	Let $\phi$ be a completely positive map from $\cT_{n,m}$ to $M_m(\C)$. Then there exists $V \in M_{mn,m}(\C)$ such that $\phi(T) =V^*TV$. Take 
	$V= (	V_1^*, 	V_2^*, \ldots 	, V_n^*)^*$, where $V_i \in M_m(\mathbb C)$. Then, for a positive element $T$ of $\cT_{n,m}$, we have 
	\begin{align*}
		\phi(T) 
			%& = \begin{pmatrix}
			%V_1^* & V_2^* & \cdots & V_n^*
		%\end{pmatrix}\begin{pmatrix}
			%Y &        &        \\
			%& \ddots &        \\
			%&        & Y
		%\end{pmatrix}
		%\begin{pmatrix}
			%I \\
			%D \\
			%\vdots \\
			%D^{\,n-1}
		%\end{pmatrix} \\
		%& \begin{pmatrix}
			%I & D^* & \cdots & (D^{\,n-1})^*
		%\end{pmatrix}
		%\begin{pmatrix}
			%Y^* &        &        \\
			%& \ddots &        \\
			%&        & Y^*
		%\end{pmatrix}
		%\begin{pmatrix}
			%V_1 \\
			%V_2 \\
			%\vdots \\
			%V_n
		%\end{pmatrix} \\
		&= \begin{pmatrix}
			V_1^*Y & V_2^*Y & \cdots & V_n^*Y
		\end{pmatrix}
		\begin{pmatrix}
			I \\
			D \\
			\vdots \\
			D^{\,n-1}
		\end{pmatrix}
		\begin{pmatrix}
			I & D^* & \cdots & (D^{\,n-1})^*
		\end{pmatrix}
		\begin{pmatrix}
			Y^*V_1 \\
			Y^*V_2 \\
			\vdots \\
			Y^*V_n
		\end{pmatrix} \\
		&=\left(\sum\limits_{i=1}^n V_i^*YD^{i-1}\right)\left(\sum\limits_{j=1}^n D^{*(j-1)}Y^*V_j\right)=\sum\limits_{i,j=1}^n V_i^*YD^{i-j}Y^*V_j. \end{align*}
	Taking $D=\text{diag}(\lambda_1, \lambda_2,\ldots,\lambda_r)$ and $Y=(y_1, y_2, \ldots, y_r)$ with $y_k \in \C^{m}$ and $\lambda_k \in \T$,
	\[YD^{i-j}Y^* = \sum\limits_{k=1}^r y_k\lambda_k^{i-j}y_k^* .\]	
	Hence, $\phi(T)=\sum\limits_{k=1}^r\sum\limits_{i,j=1}^n \lambda_k^{i-j} V_i^* y_ky_k^*V_j = \sum\limits_{k=1}^r Q_V(\lambda_k)^*y_ky_k^*Q_V(\lambda_k)$.

\smallskip
Thus, the value of $\phi(T)$ for every $T \in \cT_{n,m}$ is completely determined by evaluating the matrix-valued polynomial $Q_V$ at the points $\lambda_k \in \T$ corresponding to $T$.
	
	%\begin{defn}
		%Let $Q_V$ denote the $M_m(\mathbb C)$ valued polynomial $Q_V(z) = \sum\limits_{i=1}^n z^{n-i}V_i $.
	%\end{defn}

	%	\[
	%	Q_v(z) = \sum\limits_{k=0}^{n-1}v_k z^{n-k-1} \; \text{and}\; \tilde{Q_v}(z) = z^{n-1} \bar{Q_v(\frac{1}{\bar{z}})} = z^{n-1} \sum_{k=0}^{n-1} \bar{v}_k\, z^{n-k-1}.
	%	\]
	%	Note that if $\lambda_1, \lambda_2,\ldots, \lambda_{n-1}$ are zeroes of polynomial $Q_v$ then $\frac{1}{\bar{\lambda_1}}, \frac{1}{\bar{\lambda_2}}, \ldots, \frac{1}{\bar{\lambda_{n-1}}}$ are zeroes of $\tilde{Q_v}$.

	\section{Consequences of pureness of $\varphi$}
	
	In this section, we will prove parts (ii) and (iii) of Theorem ~\ref{TheMainTheorem} under the assumption that condition (i) holds. 
	
	%\begin{lem}\label{Vn_neq_0}
	%	Let  $\varphi:\mathcal{T}_{n,m} \to M_m(\mathbb{C})$ be a pure UCP map. Then there exists an isometry $V = (V_1, V_2, \ldots, V_n)^t$ such that $V_1 \ne 0,\;V_n \neq 0$ and $\varphi(T)=V^*TV$. Thus the degree of the polynomial $Q_V$ is the highest. 
	%\end{lem}
	%\begin{proof}
		%We already know that an isometry $W = (W_1, W_2, \ldots, W_n)^t$ exists such that  $\varphi(T)=W^*TW$. If $W_n= 0$, let $k$ be the greatest integer, $1 \le k \le n$ such that $W_k \neq 0$. Define $V$ to be the isometry such that $V_1 = V_2 = \cdots = V_{n-k} = 0, V_{n-k+1} = W_1, \ldots V_n = W_k$. Then $V_n \neq 0$ and $\varphi(T)=V^*TV$. 
	%\end{proof}
	
	% \begin{thm} \label{thm:highest}
		% 		Let  $\varphi:\mathcal{T} \to M_m(\mathbb{C})$ be a pure UCP map. Then there exists an isometry $V = (V_1, V_2, \ldots, V_n)^t$ such that $V_1 \neq 0$, $V_n \neq 0$ and $\varphi(T)=V^*TV$. Thus the degree of the polynomial $Q_V$ is the highest. 
		% {\color{red}Moreover, if there exists another isometry $W$ such that $W^*TW = V^*TV$ for all $T \in \cT$ then $V = \lambda W$ for some $|\lambda|=1$.}
		% \end{thm}
	
	% \begin{proof}
		
		% From Lemma \ref{Vn_neq_0}, get an isometry $V = (V_1 , V_2 , \ldots , V_n)^t $ such that $V_n \neq 0$ and $\varphi(T)=V^*TV$.
		
		\begin{lem}\label{Vn_neq_0}
			Let $W=(W_1, W_2, \ldots, W_n)^t $ be an isometry and $\varphi:\mathcal{T}_{n,m} \to M_m(\mathbb{C})$ be a \ucp map defined as $\varphi(T)=W^*TW$ then there exists an isometry $V = (V_1, V_2, \ldots, V_n)^t$ such that $V_n \neq 0$ and $\varphi(T)=V^*TV$.
		\end{lem}
		
		\begin{proof}
			If $W_n= 0$, let $k$ be the greatest integer, $1 \le k < n$ such that $W_k \neq 0$. Define $V$ to be the isometry such that $V_1 = V_2 = \cdots = V_{n-k} = 0, V_{n-k+1} = W_1, \ldots, V_n = W_k$. Then $V_n \neq 0$ and $\varphi(T)=V^*TV$. 
		\end{proof}
		
More can be said when the map $\varphi$ is pure.		
		
		\begin{lem}\label{V1_neq_0}
			Let  $\varphi:\mathcal{T} _{n,m} \to M_m(\mathbb{C})$ be a \ucp map defined by an isometry $V = (V_1, V_2, \ldots, V_n)^t$ with $V_n \neq 0$ such that $\varphi(T)=V^*TV$. If $\varphi$ is a pure map then $V_1 \neq 0$. Thus the degree of the polynomial $Q_V$ is the highest. 
		\end{lem}
	
		\begin{proof}
			Let  $\varphi$ be a pure \ucp map. If $V_1 = 0$, then for any block Toeplitz matrix  $T = ((T_{i-j})) \in \mathcal{T}_{n,m}$, where each $T_{k} \in M_m(\mathbb{C})$, we have
		$$\phi(T) = V^* T V = \sum_{i=2}^n \sum_{j=2}^n V_i^* T_{i-j} V_j.$$
		Shifting the indices by letting $i' = i-1$ and $j' = j-1$, the indices $i', j'$ run from $1$ to $n-1$, 
		\[ 
		\phi(T) = \sum_{i'=1}^{n-1} \sum_{j'=1}^{n-1} V_{i'+1}^* T_{i'-j'} V_{j'+1}. 
		\] 
		Consider $U \in M_{mn, m}(\mathbb{C})$ defined by $U = (V_2 , V_3 , \ldots , V_n , 0 )^t$. Then we have $U_k = V_{k+1}$ for $k \le n-1$, and $U_n = 0$. Thus, this shifted summation matches $U^* T U$ identically, i.e., 
		\[ 
		\phi(T) = V^* T V = U^* T U \quad \text{for all } T \in \mathcal{T}_{n,m}. \] 
		Define a linear map $\psi : \mathcal{T}_{n,m} \to M_m(\mathbb{C})$ by
		\[ 
		\psi(T) = \frac{1}{4} (V+U)^* T (V+U). 
		\] 
		Then, clearly $\psi$ is a completely positive map.
		Now,
		% \[ 
		% \phi(T) - \psi(T) = \frac{1}{2}V^* T V + \frac{1}{2}U^* T U - \frac{1}{4}\left( V^* T V + U^* T U + V^* T U + U^* T V \right)
		% \] 
		\[ 
		\phi(T) - \psi(T) = \frac{1}{4} (V-U)^* T (V-U). 
		\] 
		Thus $\psi \le \phi$. Now we will show that $\psi$ is not a multiple of $\phi$. Let $k \in \{2, 3, \dots, n\}$ be the smallest integer such that $V_k \neq 0$. This implies that $V_1 = V_2 = \dots = V_{k-1} = 0$ and $V_k \neq 0$. Also, for an isometry $U$ we get $U_1 = \dots = U_{k-2} = 0$ and $U_{k-1} = V_k \neq 0$.

		Fix $\ell = n - k + 1$. Choose matrix $A \in M_m(\mathbb{C})$ such that $V_n^*AV_k \neq 0$, define the block Toeplitz matrix $T^{(\ell, A)} \in \mathcal{T}_{n,m}$ by setting its block diagonals to be
		\[
		T_r = 
		\begin{cases}
			A & \text{if } r = \ell, \\
			0 & \text{if } r \neq \ell.
		\end{cases}
		\]
		By computations, we get that $\varphi(T^{(\ell, A)}) = \sum_{j=1}^{k-1} V_{j + n - k + 1}^* A V_j = 0 $ and $\psi(T^{(\ell, A)}) = \frac{1}{4} V_n^* A V_k \neq 0.$
		If $\psi$ were a scalar multiple of $\varphi$ (i.e., $\psi = t \varphi$ for some $t \in \mathbb{C}$), we would have
		$\psi(T^{(\ell, A)}) = t \varphi(T^{(\ell, A)}) = t \cdot 0 = 0,$ which contradicts $\psi(T^{(\ell, A)}) \neq 0$. Thus, $\psi$ cannot be a scalar multiple of $\varphi$, so $V_1 \neq 0$. This completes the proof.
		\end{proof}
		 	 
The two lemmas above put together give the following.

	\begin{cor}\label{cor:highest}
		Let  $\varphi:\mathcal{T}_{n,m} \to M_m(\mathbb{C})$ be a pure \ucp map. Then there exists an isometry $V = (V_1, V_2, \ldots, V_n)^t$ such that $V_1 \ne 0,\;V_n \neq 0$ and $\varphi(T)=V^*TV$. Thus the degree of the polynomial $Q_V$ is the highest.
	\end{cor}
		
The result about location of roots of $Q_V$ is now immediate by employing the method of Proposition 4.5 in ~\cite{CvS} and Proposition 2.12 of ~\cite{Hek}.

	 \begin{lem}\label{lem:roots}
		Let $\phi: \cT_{n,m} \rightarrow M_m(\C)$ be a pure \ucp map then there exists an isometry $V=(V_1,V_2,...,V_n)^t$ with $V_1, V_n \neq 0$ such that $\phi(T) = V^*TV$ and $Q_V(z) = \sum_{k=1}^n z^{n-k}V_k$ has all its roots on $\T$. 
	\end{lem}
	
	\begin{proof}
		Using Corollary ~\ref{cor:highest} there exists an isometry $V=(V_1,V_2,...,V_n)^t$ with $V_1 \neq 0, V_n \ne 0$ and $\phi(T) = V^*TV$. Clearly, $\alpha =0$ is not a root of $Q_V$ as $V_n\ne 0$. Let $\alpha$ be a non-zero root in $\C \setminus \T$, then the function $|z-\alpha|^2$ is strictly positive and non-constant on $\mathbb{T}$. Let $\delta = \min_{z \in \mathbb{T}} |z-\alpha|^2 > 0$. Choose a constant $\epsilon$ such that $0 < \epsilon < \delta$, ensuring $|z-\alpha|^2 - \epsilon > 0$ for all $z \in \mathbb{T}$.  We recall the formula ~\ref{phiFandintegral} for $\varphi$
		\[		\varphi(T) = \int_{\mathbb{T}} Q_V(z)^* F(z) Q_V(z)\; dz.
		\]
		
		Since $\alpha$ is a non-zero root we get $Q_V(z)= (z-\alpha)P(z)$ for some matrix-valued polynomial $P$ of degree atmost $n-2$, we rewrite $\varphi(T)$ as
		\[
		\varphi(T) = \int_{\mathbb{T}} |z-\alpha|^2 P(z)^* F(z) P(z) \; dz.
		\]
		Define a new map $\psi : \mathcal{T}_{n,m} \to M_m(\C)$ by
		\[
		\psi(T) = \epsilon \int_{\mathbb{T}} P(z)^* F(z) P(z)\; dz.
		\]
		Since $P(z)$ has degree at most $n-2$, its coefficients can be embedded into a column matrix $Y \in M_{mn, m}(\mathbb{C})$, making $\psi(T) = \epsilon Y^* T Y$. Thus, $\psi$ is a \cp map.
		
		Furthermore, $\phi - \psi$ is \cp because
		\begin{align*}
			(\varphi - \psi)(T) = \int_{\mathbb{T}} \left( |z-\alpha|^2 - \epsilon \right) P(z)^*  F(z) P(z) \;dz\\
			= \int_{\mathbb{T}} (r(z)P(z))^* F(z) (r(z)P(z))\; dz
		\end{align*}
		where $r(z)$ is a scalar valued polynomial such that $|r(z)|^2 =|z-\alpha|^2 - \epsilon $. Since $r(z)P(z)$ is a polynomial of degree at most $n-1$, $\varphi - \psi$ can also be written as $Z^*TZ$ and is a \cp map, meaning $\psi \le \varphi$. But clearly $\psi \neq t\varphi$ for any $t\in [0,1]$. This gives us a contradiction as $\varphi$ is a pure \ucp map.
	\end{proof}

	The crucial lemma towards uniqueness of extension is the following one. 
	\begin{lem} \label{lem:monomial}
		If there are two isometries $V = (V_1, V_2, \ldots, V_n)^t$ and $W = (W_1, W_2, \ldots, W_n)^t$ with $V_1, V_n \ne 0$ for one pure \ucp map $\varphi:\mathcal{T}_{n,m} \to M_m(\mathbb{C})$, i.e., if $\varphi(T)=V^*TV =W^*TW$ for all $T \in \mathcal{T}_{n,m}$, then $V=\lambda W$ for some $\lambda \in \mathbb{T}$.
	\end{lem}
	
	\begin{proof}
		Let $T = \left(\left( T_{j-k} \right) \right) \in \mathcal{T}_{n,m}$ be a block Toeplitz matrix of size $mn \times mn$.  So,
		\[\varphi(T) = 
		V^* T V = \sum_{j=1}^n \sum_{k=1}^n V_j^* T_{j-k} V_k
		= \sum_{\ell=-(n-1)}^{n-1} \left( \sum_{j-k=\ell} V_j^* T_\ell V_k \right).
		\]
		A similar expression holds for 	$\varphi$ in terms of $W$ also. Fix an arbitrary index $\ell \in \{-(n-1), \ldots, n-1\}$ and any $m \times m$ matrix $A$. Construct a specific block Toeplitz matrix $T^{(\ell, A)} \in \mathcal{T}_{n,m}$ defined by its diagonals
		\[
		T_r = \begin{cases} 
			A & \text{if } r = \ell \\ 
			0 & \text{if } r \neq \ell. 
		\end{cases}
		\]
		Substituting this choice of $T^{(\ell, A)}$ into our identity $V^* T^{(\ell, A)} V = W^* T^{(\ell, A)} W$ we get 
		\begin{equation} \label{VandW}
			\sum_{j-k=\ell} V_j^* A V_k = \sum_{j-k=\ell} W_j^* A W_k.
		\end{equation}
		This identity holds for every $\ell \in \{-(n-1), \ldots, n-1\}$ and for every choice of an $m \times m$ matrix $A$. Now consider the matrix-valued polynomial $Q_V(z) = \sum_{k=1}^n z^{n-k} V_k$. For any point $z$ on the unit circle $\mathbb{T}$, we have
		%\[
		%Q_V(z)^* = \left( \sum_{j=1}^n z^{n-j} V_j \right)^* = \sum_{j=1}^n \bar{z}^{n-j} V_j^* = \sum_{j=1}^n z^{j-n} V_j^*
		%\]
		%Expanding the product $Q_V(z)^* A Q_V(z)$ for a constant matrix $A$, we get 
		%\begin{align*}
		%	Q_V(z)^ * A Q_V(z) &= \left( \sum_{j=1}^n z^{j-n} V_j^* \right) A \left( \sum_{k=1}^n z^{n-k} V_k \right) \\
		%	&= \sum_{j=1}^n \sum_{k=1}^n z^{j-n} z^{n-k} V_j^* A V_k \\
		%	&= \sum_{j=1}^n \sum_{k=1}^n z^{j-k} V_j^* A V_k
		%\end{align*}
	
	\[	
			Q_V(z)^* A Q_V(z) = \sum_{j=1}^n \sum_{k=1}^n z^{j-k} V_j^* A V_k\\
			=  \sum_{\ell=-(n-1)}^{n-1} z^\ell \left( \sum_{j-k=\ell} V_j^* A V_k \right).
	\]
		%Just as we did previously, we can gather these terms by the powers of $z$ using the substitution $\ell = j - k$ to get 
		%\[
		%Q_V(z)^* A Q_V(z) = \sum_{\ell=-(n-1)}^{n-1} z^\ell \left( \sum_{j-k=\ell} V_j^* A V_k \right).
		%\]
		Applying identical steps to the polynomial $Q_W(z)$ results in
		\[
		Q_W(z)^* A Q_W(z) = \sum_{\ell=-(n-1)}^{n-1} z^\ell \left( \sum_{j-k=\ell} W_j^* A W_k \right).
		\]
		By ~\eqref{VandW}, we get $Q_V(z)^* A Q_V(z) = Q_W(z)^* A Q_W(z) $ for all $ A \in M_m(\C)$ and all $z \in \T$. 			
		Clearly, for \(z_0\in\mathbb{T}\), $ Q_V(z_0)=0$ if and only if $Q_W(z_0)=0.$
		Letting $A=E_{k\ell}$ and comparing the $ab$-th entry we get 
	%Let \(A=E_{k\ell}\). Then the \(ab\)-th entry is the same
		%\[
		%\bigl(Q_V(z)^*E_{k\ell}Q_V(z)\bigr)_{ab}
		%=
		%\bigl(Q_W(z)^*E_{k\ell}Q_W(z)\bigr)_{ab}.
		%\]
		%Hence,
		\begin{equation}\label{kalb}
			\overline{Q_V(z)_{ka}}\,Q_V(z)_{\ell b}
			=
			\overline{Q_W(z)_{ka}}\,Q_W(z)_{\ell b},
			\qquad \forall z\in\mathbb{T},
		\end{equation}
		for all $\ell,k,a,b\in\{1,\dots,m\}$.\\
		Since $Q_W(z)$ is a non-zero matrix-valued polynomial, there exist $i,j$ such that $ Q_W(z)_{ij}\neq 0$, except at finitely many points on $\mathbb{T}$, since $Q_W(z)_{ij}$ is a polynomial in $z$.
		
		Putting $k=i,\ a=j,\ \ell=i,\ b=j$ in ~\eqref{kalb}, we get
		\[
		\overline{Q_V(z)_{ij}}\,Q_V(z)_{ij}
		=
		\overline{Q_W(z)_{ij}}\,Q_W(z)_{ij},
		\qquad \forall z\in\mathbb{T}.
		\]
		Therefore, $|Q_V(z)_{ij}|=|Q_W(z)_{ij}|	\neq 0$ except at finitely many points on $\mathbb{T}$. Hence,
		\begin{equation}\label{rational}
			\left|
			\frac{Q_V(z)_{ij}}{Q_W(z)_{ij}}
			\right|
			=1.
		\end{equation}
		Define the rational function $ 	R:\mathbb{C}\to\widehat{\mathbb{C}} =\mathbb{C}\cup\{\infty\}$ 
		by
		\begin{equation} \label{R}
			R(z)=\frac{Q_V(z)_{ij}}{Q_W(z)_{ij}}.
		\end{equation}	
		From ~\eqref{kalb}, we have $\overline{Q_V(z)_{ij}}\,Q_V(z)_{\ell b} = \overline{Q_W(z)_{ij}}\,Q_W(z)_{\ell b}$ for all $z\in\mathbb{T}.$\\
		%Hence,
		%\[
		%|Q_V(z)_{ij}|^2\,Q_V(z)_{\ell b}
		%=
		%Q_V(z)_{ij}\,
		%\overline{Q_W(z)_{ij}}\,
		%Q_W(z)_{\ell b}.
		%\]
		%Therefore,
		%\[
		%|Q_W(z)_{ij}|^2\,Q_V(z)_{\ell b}
		%=
		%Q_V(z)_{ij}\,
		%\overline{Q_W(z)_{ij}}\,
		%Q_W(z)_{\ell b}.
		%\]
		Thus, we get that $Q_W(z)_{ij}\,Q_V(z)_{\ell b} =Q_V(z)_{ij}\,Q_W(z)_{\ell b}$ for all except at finitely many points in \(\mathbb{T}\). Hence,
		\[
		Q_V(z)_{\ell b}
		=
		\frac{Q_V(z)_{ij}}{Q_W(z)_{ij}}\,
		Q_W(z)_{\ell b}
		\]
		except at finitely many points. Let $F:\mathbb{C}\to\mathbb{C}$ be a function defined by
		\[
		F(z)
		=
		Q_W(z)_{ij}\,Q_V(z)_{\ell b}
		-
		Q_V(z)_{ij}\,Q_W(z)_{\ell b}.
		\]
		Then
		\begin{enumerate}
			\item[(i)] $F$ is a polynomial function on $\mathbb{C}$, since
			$Q_V(z)_{ij}$ and $Q_W(z)_{\ell b}$ are polynomials defined on $\mathbb C$.
			
			\item[(ii)] $F(z)=0$ on $\mathbb{T}$ except at finitely many points.
		\end{enumerate}
		A polynomial having infinitely many roots is identically zero. Therefore, $F = 0$.
		%\[
		%F(z)=0
		%\qquad \forall z\in\mathbb{C}.
		%\]		
		This implies that $Q_W(z)_{ij}\,Q_V(z)_{\ell b} = Q_V(z)_{ij}\,Q_W(z)_{\ell b}$ for all $z\in\mathbb{C}.$
		Therefore,
		\[
		Q_V(z)_{\ell b}
		=
		\frac{Q_V(z)_{ij}}{Q_W(z)_{ij}}\,
		Q_W(z)_{\ell b},
		\qquad \forall z\in\mathbb{C},
		\]
		except at those points where $ Q_W(z)_{ij}=0.$ Hence, $ Q_V(z)=R(z)\,Q_W(z)$ where $R$ is as defined in ~\eqref{R}. Also, by ~\eqref{rational}, $|R(z)|=1$
		on \(\mathbb{T}\), except at finitely many points. Since $R$ is a rational function, by using limiting arguments, we get $|R(z)|=1$ on $\mathbb{T}$.

		\smallskip
		Hence,  $\exists \; \alpha_j \in \mathbb D$ for $j \in \{1,\ldots,N\}$, $\beta_s \in \mathbb{D}$ for $s \in \{1,\ldots,M \}$ and $m \in \Z$ such that
		\[
		R(z)=e^{i\theta}z^m \left( \prod\limits_{j=1}^N 
		\frac{z-\alpha_j}{1-\bar{\alpha_j}z}\right)\left( \prod\limits_{s=1}^M \frac{1-\bar{\beta_s}z}{z-\beta_s}\right).
		\]
		And the following equation holds on $\C$ 
		
		\[
		Q_V(z)= e^{i\theta}z^m \left( \prod\limits_{j=1}^N 
		\frac{z-\alpha_j}{1-\bar{\alpha_j}z}\right)\left( \prod\limits_{s=1}^M \frac{1-\bar{\beta_s}z}{z-\beta_s}\right) Q_W(z).
		\]
		Since $V_1 \neq 0$, we have deg$(Q_V) =n-1$ and deg$(Q_W)\le n-1$.  This will imply that $m \ge 0$.
		Since $\phi$ is a pure \ucp map so following the same arguments as done in Lemma ~\ref{lem:roots} it can be shown that such $\alpha_j$ and $\beta_s$ will not exist in our case.

		So we have $Q_V(z) = e^{i\theta}z^mQ_W(z)$. Moreover, using $V_n \ne 0$, we get that $m=0$. Hence, $Q_V(z) = e^{i\theta}Q_W(z)$ and this proves the assertion that $V=\lambda W$.
		\end{proof}

       \begin{thm}
		A pure \ucp map $\phi : \cT_{n,m} \to M_m(\C)$ has a unique \cp extension to $M_{mn}(\C)$.
	\end{thm}
	
	\begin{proof}
		Let $\phi : \cT_{n,m} \to M_m(\C)$ be a pure \ucp map. Get a $V$ as in Corollary ~\ref{cor:highest}. Clearly, $\tilde{\phi}(T) =V^*TV$ is a \cp extension of $\phi$ from $\cT_{n,m}$ to $M_{mn}(\C). $ Let $\phi_1$ be another \cp extension of map $\phi$. By Choi's theorem in ~\cite{Choi}, there exist $mn \times m$ matrices $W_1, \ldots , W_k$ such that $\phi_1(T)= \sum\limits_{i=1}^k W_i^*TW_i$.
		
		This implies that for all $i$, $W_i^*TW_i \leq V^*TV$ and $V^*TV $ is a pure map so there exists $t_i \in (0,1]$ such that $t_i V^*TV = W_i^*TW_i $. Also, $\sum t_i = 1$ because $\sum\limits_{i=1}^k W_i^*TW_i = V^*TV$ for $t \in \mathcal T_{n,m}$. Now using Lemma ~\ref{lem:monomial}, we get $W_i = \lambda_i t_i^{1/2} V$ where $\lambda_i \in \mathbb T$. Putting these values back,
		\[
		\phi_1(T) = \sum_{i=1}^k t_i V^*TV = V^*TV = \phi(T), \qquad T \in M_{mn}(\C).
		\]
		Hence, a pure map $\phi$ has a unique \cp extension property.
	\end{proof}

	Next we quote a folklore result from Arveson's boundary theory, and it follows from Arveson's extension theorem.

	\begin{thm} \label{thm:converse}
		Let $\mathcal{A}$ be a unital $C^*$-algebra, and let $\mathcal{S} \subseteq \mathcal{A}$ be an operator subsystem. Let $\mathcal{H}$ be a Hilbert space and $\mathcal{B}(\mathcal{H})$ be the algebra of bounded linear operators on $\mathcal{H}$.
		
		Let $\Phi: \mathcal{A} \to \mathcal{B}(\mathcal{H})$ be a pure completely positive (\it{cp}) map, and let $\phi = \Phi|_\mathcal{S}$ be its restriction to the operator subsystem $\mathcal{S}$. If $\Phi$ is the unique \cp extension of $\phi$ to $\mathcal{A}$, then $\phi$ is a pure \cp map on $\mathcal{S}$.
	\end{thm}

Thus, the proof of (ii) implying (i) in Theorem ~\ref{TheMainTheorem} is also complete.

	\section{The roots of the polynomial}
	
Now that (i) and (ii) in Theorem ~\ref{TheMainTheorem} are equivalent and (i) implying (iii) has already been proved, we take up the proof of (iii) implying (ii) below.

	\begin{thm} \label{thm:unique_CP}
		Let $V=(V_1,V_2,\ldots, V_n)^t$ be an isometry such that $V_1, V_n \neq 0$ and $Q_V(z) = \sum_{k=1}^n z^{n-k} V_k$ has all its roots (i.e., all $a\in \C$ such that $Q_V(a)$ is a zero matrix) on $\T$. Then $\phi : \cT_{n,m} \to M_m(\C)$ defined by $\phi(T) = V^*TV$ has a unique \ucp extension to $M_{mn}(\mathbb{C})$; that is, the only \cp map $\Phi: M_{mn}(\mathbb{C}) \to M_m(\mathbb{C})$ satisfying $\Phi(T) = \phi(T)$ for all $T \in \cT_{n,m}$ is the canonical map $\Phi(X) = V^* X V$.
	\end{thm}
	
	Before proving the theorem, we establish a useful lemma regarding identical completely positive maps with low-rank Kraus representations. The lemma must be known and can be proved in different ways. 
	
	\begin{lem}\label{lem:kraus}
		Let $X, X_1, \dots, X_N \in M_m(\mathbb{C})$ be matrices such that for all $A \in M_m(\mathbb{C})$,
		\begin{equation}\label{eq:lemma_hyp}
			\sum_{r=1}^N X_r^* A X_r = X^* A X.
		\end{equation}
		Then there exist scalars $c_1, \dots, c_N \in \mathbb{C}$ such that $X_r = c_r X$ for each $r = 1, \dots, N$ and $\sum_{r=1}^N |c_r|^2 = 1$.
	\end{lem}

	\begin{proof}

The hypothesis says the two \cp maps
\[
\phi(A) = X^* A X, \qquad \psi(A) = \sum_{r=1}^N X_r^* A X_r
\]
on $M_m(\mathbb{C})$ are equal. Choi's uniqueness theorem (see ~\cite{Choi}) says any two Kraus decompositions of the same \cp map are related by an isometry acting on the ``Kraus index'': padding $\{X\}$ to $\{X, 0, \ldots, 0\}$ of length $N$, there exists an $N \times N$ unitary $(u_{rs})$ with
\[
X_r = u_{r1} X + \sum_{s \ge 2} u_{rs} \cdot 0 = u_{r1} X.
\]
So $c_r := u_{r1}$ works, and $\sum_r |c_r|^2 = 1$ because $(u_{r1})_r$ is a unit column of a unitary. 
 
Equivalently in Arveson's language: since $\phi(A) = X^*AX$ is pure; Arveson's Radon--Nikodym theorem says every \cp map dominated by $\phi$ is a scalar multiple $c\phi$, and $A \mapsto X_r^* A X_r$ is dominated by $\phi$. Getting $X_r = c_r X$ (rather than only $X_r^* A X_r = |c_r|^2 X^* A X$) requires the extra step that the Radon--Nikodym derivative lives in $\pi(M_m)' = \mathbb{C} I$.
	\end{proof}
	
	We are now ready to present the full proof of the theorem.
	
	\begin{proof}[Proof of Theorem ~\ref{thm:unique_CP}]
		Let $\Phi: M_{mn}(\mathbb{C}) \to M_m(\mathbb{C})$ be an arbitrary \cp extension of $\phi$. By Choi's Theorem in ~\cite{Choi}, there exist matrix operators $W_r \in M_{mn, m}(\mathbb{C})$ for $r = 1, \dots, N$ such that
		\begin{equation}\label{eq:phi_kraus}
			\Phi(X) = \sum_{r=1}^N W_r^* X W_r, \quad \forall X \in M_{mn}(\mathbb{C}).
		\end{equation}
		We can partition each Kraus operator $W_r$ into $W_r = (W_{1,r}, W_{2,r}, \ldots, W_{n,r})^t$ 
		%\[
		%W_r = \begin{pmatrix} W_{1,r} \\ W_{2,r} \\ \vdots \\ W_{n,r} \end{pmatrix}
		%\]
		where each $W_{k,r} \in M_m(\mathbb{C})$. Since $\Phi$ extends $\phi$, we must have $\Phi(T) = \phi(T)$ for all $T \in \cT_{n,m}$. Writing this explicitly in terms of the block structures yields
		\begin{equation}\label{eq:toeplitz_match}
			\sum_{r=1}^N \sum_{j,k=1}^n W_{j,r}^* T_{j-k} W_{k,r} = \sum_{j,k=1}^n V_j^* T_{j-k} V_k
		\end{equation}
		for every block Toeplitz matrix $T = ((T_{j-k})) \in \cT_{n,m}$. By choosing $T$ to be a block Toeplitz matrix that has a single non-zero block $A \in M_m(\mathbb{C})$ on its $\ell$-th diagonal and zeroes everywhere else, ~\eqref{eq:toeplitz_match} simplifies to
		\begin{equation}\label{eq:diagonal_match}
			\sum_{r=1}^N \sum_{j-k=\ell} W_{j,r}^* A W_{k,r} = \sum_{j-k=\ell} V_j^* A V_k, \quad \forall A \in M_m(\mathbb{C}).
		\end{equation}
		For each $r=1, \ldots, N$, define the matrix-valued polynomial $Q_{W_r}(z)=\sum\limits_{k=1}^n z^{n-k} W_{k,r}$.\\ 
		Let $z$ be a complex number on the unit circle $\mathbb{T}$. Consider the expression $Q_V(z)^* A Q_V(z)$ for an arbitrary $A \in M_m(\mathbb{C})$. Expanding it gives
		\begin{align*}
			Q_V(z)^* A Q_V(z) &= \sum_{\ell=-(n-1)}^{n-1} z^\ell \left( \sum_{j-k=l} V_j^* A V_k \right).
		\end{align*}
		Applying the diagonal relation from ~\eqref{eq:diagonal_match} to the inner sum, we obtain
		%\begin{align*}
			%Q_V(z)^* A Q_V(z) &= \sum_{\ell=-(n-1)}^{n-1} z^l \left( \sum_{r=1}^N %\sum_{j-k=\ell} W_{j,r}^* A W_{k,r} \right) \\
			%&= \sum_{r=1}^N \sum_{j,k=1}^n z^{j-k} W_{j,r}^* A W_{k,r} \\
			%&= \sum_{r=1}^N Q_{W_r}(z)^* A Q_{W_r}(z)
		%\end{align*}
		%Thus, for all $z \in \mathbb{T}$ and all $A \in M_m(\mathbb{C})$, we have the fundamental identity
		\begin{equation}\label{eq:poly_identity}
			\sum_{r=1}^N Q_{W_r}(z)^* A Q_{W_r}(z) = Q_V(z)^* A Q_V(z)
		\end{equation}
		for all $z \in \mathbb{T}$ and all $A \in M_m(\mathbb{C})$.
		
		Fix any $z \in \mathbb{T}$. By applying Lemma ~\ref{lem:kraus} directly to ~\eqref{eq:poly_identity} with $X_r = Q_{W_r}(z)$ and $X = Q_V(z)$, there exist scalar coefficients $c_r(z) \in \mathbb{C}$ such that
			\begin{equation}\label{eq:pointwise_scalar}
				Q_{W_r}(z) = c_r(z) Q_V(z), \quad \forall r = 1, \dots, N
			\end{equation}
			and these coefficients satisfy
			\begin{equation}\label{eq:circle_norm}
				\sum_{r=1}^N |c_r(z)|^2 = 1, \quad \forall z \in \mathbb{T}.
			\end{equation}
			
			Since $V_1 \neq 0$, the matrix-valued polynomial $Q_V(z)$ has degree exactly $n-1$ and is not identically zero. Therefore, there exists indices $i,j$ such that the $(i,j)$-th entry of $Q_V$, denoted by $p(z) = Q_V(z)_{ij}$, is a non-zero scalar polynomial. From ~\eqref{eq:pointwise_scalar}, for any $z \in \mathbb{T}$ where $p(z) \neq 0$, we can express $c_r(z)$ as
			\begin{equation*}
				c_r(z) = \frac{Q_{W_r}(z)_{ij}}{Q_V(z)_{ij}} = \frac{Q_{W_r}(z)_{ij}}{p(z)}.
			\end{equation*}
			Putting this value in ~\eqref{eq:pointwise_scalar}, we get for $z \in \T$ where $p(z) \neq 0$, 
			\begin{equation*}
				Q_{W_r}(z) = \frac{Q_{W_r}(z)_{ij}}{p(z)} Q_V(z).
			\end{equation*}
			So, following the same arguments as in the proof of Lemma ~\ref{lem:monomial} we get that $ p(z) Q_{W_r}(z) = Q_{W_r}(z)_{ij} Q_V(z)$ for all $z \in \C.$
			% This demonstrates that $c_r(z)$ is a rational function restricted to the unit circle. By identity theorems for meromorphic functions, $c_r(z)$ extends uniquely to a rational function over the entire complex plane $\mathbb{C}$.
			Let us write each rational function defined on $\C$ in lowest terms as $$f_r(z) =\frac{Q_{W_r}(z)_{ij}}{p(z)}  = \frac{A_r(z)}{B(z)}$$ where $A_1(z), \ldots, A_N(z)$ and $B(z)$ are polynomials sharing no common polynomial factor of degree $\ge 1$. Substituting this into ~\eqref{eq:pointwise_scalar} and clearing denominators gives
			\begin{equation}\label{eq:poly_equality_all}
				B(z) Q_{W_r}(z) = A_r(z) Q_V(z), \quad \forall z \in \mathbb{C}, \, \forall r = 1, \ldots, N.
			\end{equation}
		Suppose for contradiction that $B(z)$ is not a constant polynomial. Then there must exist some root $z_0 \in \mathbb{C}$ such that $B(z_0) = 0$. Evaluating ~\eqref{eq:poly_equality_all} at $z = z_0$ gives
		\[
		0 = A_r(z_0) Q_V(z_0), \quad \forall r = 1, \dots, N.
		\]
		Because the set of polynomials $\{A_1, \dots, A_N, B\}$ shares no common roots, it is impossible for all $A_r(z_0)$ to equal zero simultaneously. Thus, there exists an index $r_0$ such that $A_{r_0}(z_0) \neq 0$ giving
		\begin{equation*}
			Q_V(z_0) = 0.
		\end{equation*}
		By the hypothesis of the theorem, all points where $Q_V(z)$ vanishes must lie on the unit circle $\mathbb{T}$. Thus, $z_0 \in \mathbb{T}$. Since on $\T,\; f_r(z)=c_r(z)$. From ~\eqref{eq:circle_norm}, the relation $\sum_{r=1}^N |A_r(z)|^2 = |B(z)|^2$ holds identically on $\mathbb{T}$ by continuity. Evaluating this at $z = z_0 \in \mathbb{T}$ yields
		\[
		\sum_{r=1}^N |A_r(z_0)|^2 = |B(z_0)|^2 = 0 \implies A_r(z_0) = 0, \qquad      \forall r = 1, \ldots, N.
		\]
		This directly contradicts the fact that $\gcd(A_1, \dots, A_N, B) = 1$. Consequently, $B(z)$ cannot possess any roots, meaning $B(z) = b \in \mathbb{C} \setminus \{0\}$ is a non-zero constant. Thus, each $f_r(z) = \frac{A_r(z)}{b}$ is a polynomial.
		
		\smallskip
		We now bound the degree of the polynomials $f_r(z)$. Since $\deg(Q_V) = n-1$ and $\deg(Q_{W_r}) \le n-1$. If $f_r(z)$ are non-zero polynomial of degree $d_r \ge 1$, then the degree of the product on the right-hand side of ~\eqref{eq:pointwise_scalar} would be $\deg(f_r Q_V) = d_r + \deg(Q_V) = d_r + n -1$. Equating degrees in $Q_{W_r}(z) = f_r(z) Q_V(z)$ gives $d_r + n - 1 \le n - 1$ imply $d_r \le 0.$ This implies that $d_r = 0$, meaning each $f_r(z) = c_r$ is a constant scalar.
		
		Since $f_r(z) = c_r$ is constant, we can equate coefficients of $z^{n-k}$ for each $k = 1, \dots, n$ in the polynomial identity
		\[
		\sum_{k=1}^n z^{n-k} W_{k,r} = c_r \sum_{k=1}^n z^{n-k} V_k.
		\]
		This gives us that $W_{k,r} = c_r V_k$  for all $k = 1, \dots, n$ and for all $r = 1, \ldots, N$. We conclude that for each $r$, $W_r = c_r V $.
		
		Finally, substituting this expression back into the Choi's representation of $\Phi$ given by ~\eqref{eq:phi_kraus} for an arbitrary matrix $X \in M_{mn}(\mathbb{C})$, we find
		\[
		\Phi(X) = \sum_{r=1}^N (c_r V)^* X (c_r V) = \left( \sum_{r=1}^N |c_r|^2 \right) V^* X V.
		\]
		Since $\sum_{r=1}^N |c_r|^2 = 1$ from ~\eqref{eq:circle_norm}, this simplifies to $\Phi(X) = V^* X V.$ This completes the proof.
	\end{proof}

  \section{Hausdorff Convergence}

\subsection{Proof of Theorem ~\ref{thm:existenceQ_V}}

The following lemma is a generalization of [~\cite{Arveson69}, Theorem 1.4.10]. It may be a folklore. We did not find the exact reference and hence give a short proof.
\begin{lem}
	The extreme points of $\mathcal{Y}_m$, denoted $\text{Ext}(\mathcal{Y}_m)$, consist exactly of discrete \ucp maps $\Psi$ supported on a finite subset $\{t_1, \dots, t_k\} \subset \mathbb{T}$ of the form
	\[
	\Psi(f) = \sum_{j=1}^k \sum_{\alpha=1}^{n_j} A_{j,\alpha}^* f(t_j) A_{j,\alpha},
	\]
	where $A_{j,\alpha} \in M_m(\C)$ and satisfying $\sum_{j,\alpha} A_{j,\alpha}^* A_{j,\alpha} = I_m$. Consequently, any $\Phi \in \mathcal{Y}_m$ can be approximated in the $\rho$-metric by a finite convex combination of such extreme points.
\end{lem}

\begin{proof}
	 Let $\Psi \in \mathcal{Y}_m$. By Stinespring's Dilation Theorem, $\Psi(f) = W^* \pi(f) W$ where $\pi$ is a unital $*$-representation of $C(\mathbb{T}, M_m(\mathbb{C}))$ on a Hilbert space $\mathcal{H}$, and $W: \mathbb{C}^m \to \mathcal{H}$ is an isometry. By [~\cite{Arveson69},Theorem 1.4.6], $\Psi$ is extreme if and only if the map $\mathcal{F}: \pi(C(\mathbb{T}, M_m(\mathbb{C})))' \to M_m(\mathbb{C})$ given by $\mathcal{F}(T) = W^* T W$ is injective. This forces the commutant to be finite-dimensional which cannot happen unless $\pi(C(\mathbb{T}, M_m(\mathbb{C})))$ is finite-dimensional. This means that the representation $\pi$ is a direct sum of irreducible finite-dimensional representations. The only irreducible finite-dimensional representations of $C(\mathbb{T}, M_m(\mathbb{C}))$ are point evaluations. Thus $\pi$ is supported on a finite set of points $\{t_1, \dots, t_k\} \subset \mathbb{T}$. So, $\mathcal{H} = \oplus_{j=1}^k \mathcal{H}_j\; \text{where}\; \mathcal{H}_j \cong \C^m \otimes \C^{n_j}$ and on each component space $\mathcal{H}_j$, the representation acts as
	 \[
	  \pi(f)|_{\mathcal{H}_j}= f(t_j) \otimes I_{n_j}, \; \text{for}\; f \in C(\T).
	 \]
	 Decomposing the isometry $W$ into matrix operators
	 \[ W: C^m \to \mathcal{H} = \oplus_{j=1}^k (\C^m \otimes \C^{n_j} )
	 \]
	 we first get $W_j: \C^m \to \C^m \otimes \C^{n_j}$. The $A_{j, \alpha} \in M_m(\C)$, for $\alpha = 1,2, \ldots , n_j$ are the ``components" of $W_j$. Substituting $\pi(f)$ and $W$
 	back into Stinespring's formula we get the required form of extreme points.		
		 
		Since $\mathcal{Y}_m$ is a compact, convex, metrizable topological space, the Krein-Milman theorem (see ~\cite{WW}) gives that $\mathcal{Y}_m = \overline{\text{co}(\text{Ext}(\mathcal{Y}_m))}$. Hence, any map $\Phi \in \mathcal{Y}_m$ can be approximated arbitrarily closely by a finite convex combination of elements from $\text{Ext}(\mathcal{Y}_m)$.
\end{proof}

The elements of the finite subset of the unit circle in the above result can be made distinct. Recall $\mathcal L$ from Notation ~\ref{L} and $\rho$ from Definition ~\ref{rho}.
\begin{lem}[Point-Splitting]
	Let $\Phi = \sum_{p=1}^L \lambda_p \Psi_p$ be an arbitrary finite convex combination of extreme points in $\mathcal{Y}_m$. Then for any $\epsilon > 0$, there exists a \ucp map $\Phi_0 \in \mathcal{Y}_m$ such that
	\[
	\Phi_0(f) = \sum_{i=1}^N B_i^* f(x_i) B_i,
	\]
	where the points $x_1, x_2, \dots, x_N \in \mathbb{T}$ are distinct and $\rho(\Phi, \Phi_0) < \epsilon$.
\end{lem}

\begin{proof}
	By expanding the definitions of the extreme points $\Psi_p$ and absorbing the convex weights $\lambda_p$, collect all overlapping evaluation points into a single set of unique distinct points $\{T_1, T_2, \dots, T_R\} \subset \mathbb{T}$. The combination $\Phi$ can be written exactly as
	\[
	\Phi(f) = \sum_{r=1}^R \sum_{\beta=1}^{M_r} C_{r,\beta}^* f(T_r) C_{r,\beta}.
	\]
	% where $T_1, \dots, T_R$ are strictly distinct
	% and the Kraus operators satisfy the localized allocation constraints $\sum_{\beta=1}^{M_r} C_{r,\beta}^* C_{r,\beta} = \gamma_r I_m$ with $\sum_{r=1}^R \gamma_r = 1$, ensuring $\sum_{r=1}^R \sum_{\beta=1}^{M_r} C_{r,\beta}^* C_{r,\beta} = I_m$.
	
	For each point $T_r$, we select $M_r$ distinct points $\{x_{r,1}, x_{r,2}, \dots, x_{r,M_r}\} \subset \mathbb{T}$ within a radius of $\delta > 0$ around $T_r$, such that $|x_{r,\beta} - T_r| < \delta$. We choose $\delta$ small enough to prevent them from overlapping and ensuring that all $N = \sum_{r=1}^R M_r$ elements in the set $\{x_{r,\beta}\}$ are distinct. We define the perturbed map $\Phi_0$ as
	\[
	\Phi_0(f) = \sum_{r=1}^R \sum_{\beta=1}^{M_r} C_{r,\beta}^* f(x_{r,\beta}) C_{r,\beta}.
	\]
	Clearly, $\Phi_0(I_m) = \sum_{r, \beta} C_{r,\beta}^* C_{r,\beta} = I_m$. Now, let $g \in \mathcal{L}$. 
	\begin{align*}
		\|\Phi(g) - \Phi_0(g)\|_{M_m(\mathbb{C})} &= \left\| \sum_{r=1}^R \sum_{\beta=1}^{M_r} C_{r,\beta}^* \big(g(T_r) - g(x_{r,\beta})\big) C_{r,\beta} \right\|_{M_m(\mathbb{C})} \\
		&\le \sum_{r=1}^R \sum_{\beta=1}^{M_r} \|C_{r,\beta}\|^2 \cdot \|g(T_r) - g(x_{r,\beta})\|_{M_m(\mathbb{C})}.
	\end{align*}
	Since $g \in \mathcal{L}$, its Lipschitz constant is bounded by 1, giving $\|g(T_r) - g(x_{r,\beta})\|_{M_m(\mathbb{C})} \le |T_r - x_{r,\beta}| < \delta$. Substituting this bound gives
	\[
	\|\Phi(g) - \Phi_0(g)\|_{M_m(\mathbb{C})} \le \delta \sum_{r=1}^R \sum_{\beta=1}^{M_r} \|C_{r,\beta}\|^2.
	\]
	Because this upper bound is uniform and entirely independent of the choice of $g$, it bounds the supremum
	\[
	\rho(\Phi, \Phi_0) \le \delta \sum_{r=1}^R \sum_{\beta=1}^{M_r} \|C_{r,\beta}\|^2.
	\]
	By specifying $\delta < \epsilon \left(\sum_{r,\beta} \|C_{r,\beta}\|^2\right)^{-1}$, we guarantee that $\rho(\Phi, \Phi_0) < \epsilon$.
\end{proof}

In fact, the operators $B_i$ can be chosen to be invertible by making use of an interesting result from linear algebra.

\begin{lem}\label{lem:invertible}
	Given any $\epsilon >0$ and a \ucp map $\Phi_0(f) = \sum_{i=1}^N B_i^* f(x_i) B_i$ with distinct points $x_i \in \mathbb{T}$, there exists a \ucp map 
	\[\Phi'(f) = \sum_{i=1}^N \lambda_i A_i^* f(x_i) A_i\] 
	such that $\lambda_i > 0$, $A_i \in SL_m(\mathbb{C})$, and $\rho(\Phi_0, \Phi') < \epsilon$. Furthermore, there exists a matrix-valued polynomial $P(z) \in M_m(\mathbb{C}[z])$ such that $P(x_i) = A_i$ for all $i$, and $\det(P(z)) = 1$ identically for all $z \in \mathbb{C}$.
\end{lem}

\begin{proof}
		The general linear group $GL_m(\mathbb{C})$ is open and dense in $M_m(\mathbb{C})$. We can select invertible matrices $\tilde{B}_{i,k} \in GL_m(\mathbb{C})$ arbitrarily close to $B_i$. Let \[S_k= \sum_{i=1}^N \tilde{B}_{i,k}^* \tilde{B}_{i,k}.\] As $\tilde{B}_{i,k} \to B_i$, for sufficiently large $k> k_0$, $S_k$ is strictly positive definite. Define the normalized matrices $A_{i,k}' = \tilde{B}_{i,k} S_k^{-1/2}$, which satisfy $\sum_{i=1}^N (A_{i,k}')^*  A_{i,k}' = I_m$ and $A_{i,k}' \to B_{i}$. Now let $\Phi'$ be the map defined using $A_i'=A_{i,k}'$, where $k$ is chosen large enough such that $\rho(\Phi', \Phi_0)< \epsilon$.
		
		Let $c_i = \det(A_i')^{1/m} \in \mathbb{C}\setminus\{0\}$ and define $A_i = c_i^{-1} A_i' \in SL_m(\mathbb{C})$, so $\det(A_i) = 1$. Setting $\lambda_i = |c_i|^2 > 0$, the map is written exactly as $\Phi'(f) = \sum_{i=1}^N \lambda_i A_i^* f(x_i) A_i$.
		
		We now construct the matrix-valued polynomial $P(z)$. A basic theorem of linear algebra states that $SL_m(\mathbb{C})$ is generated by elementary matrices $E_{j,k}(c) = I_m + c \cdot e_{j,k}$ ($j \neq k$), see [~\cite{Artin}, Page 71]. Every $A_i$ can be factored into a finite product of elementary matrices. By inserting identity elements $E_{j,k}(0) = I_m$ we can combine factorization into one product
		\[
		A_i = \prod_{\alpha=1}^M E_{j_\alpha, k_\alpha}(c_{i, \alpha}) \quad \forall i \in \{1, \dots, N\}.
		\]
		For each fixed index $\alpha \in {1, \dots, M}$, applying classical scalar Lagrange interpolation to get the scalar values $\{c_{1,\alpha}, \dots, c_{N,\alpha}\}$ at the distinct points $x_1, \dots, x_N \in \mathbb{T}$. This yields a scalar polynomial $p_\alpha(z) \in \mathbb{C}[z]$ satisfying $p_\alpha(x_i) = c_{i,\alpha}$. We define a matrix-valued polynomial
		\[
		P(z) = \prod_{\alpha=1}^M E_{j_\alpha, k_\alpha}(p_\alpha(z)).
		\]
		Since each factor satisfies $\det(E_{j_\alpha, k_\alpha}(p_\alpha(z))) = 1$, gives us that $\det(P(z)) = 1 \quad \forall z \in \mathbb{C}.$
		Evaluating at the points gives $P(x_i) = \prod_{\alpha=1}^M E_{j_\alpha, k_\alpha}(p_\alpha(x_i)) = \prod_{\alpha=1}^M E_{j_\alpha, k_\alpha}(c_{i,\alpha}) = A_i$.
	\end{proof}

\textsc{We are now ready to complete the proof of Theorem ~\ref{thm:existenceQ_V}.}

\smallskip

 Because of the reduction made in Lemma ~\ref{lem:invertible}, we only need to consider $\Phi' \in \mathcal{Y}_m$ of the form 
	$
	\Phi'(f)=\sum\limits_{i=1}^N \lambda_i A_i^* f(x_i) A_i
	$
	such that $\lambda_i > 0$, $A_i \in SL_m(\mathbb{C})$ and $\{x_1, \cdots, x_N \}$ are all distinct points in $\mathbb{T}$.
	
	Let $\mu = \sum\limits_{i=1}^N \lambda_i \delta_{x_i}$ be the positive discrete measure on $\mathbb{T}$. Then from [~\cite{Hek}, Proposition 3.5], there exists a sequence of scalar polynomials $q_n(z) \in \mathbb{C}[z]$ such that all roots of $q_n(z) = 0$ lie strictly on $\mathbb{T}$ and the measures $d\mu_n(z) = |q_n(z)|^2 \;dz$ converge weakly to $\mu$ as $n \to \infty$.
	
	\smallskip
	Let $P(z)$ be the matrix-valued polynomial from Lemma ~\ref{lem:invertible}. Define the sequence of matrix-valued polynomials
	\[
	\tilde{Q}_n(z) = q_n(z) P(z) \in M_m(\C[z]).
	\]
	Clearly, the roots of $\tilde{Q}_n$ are on $\T$. Now, consider the sequence of the induced map  $\tilde{\Phi}_n(f) = \int_{\mathbb{T}} \tilde{Q}_n^*(z) f(z) \tilde{Q}_n(z) \, dz$. Expanding the terms gives
	\[
	\tilde{\Phi}_n(f) = \int_{\mathbb{T}} |q_n(z)|^2 P(z)^* f(z) P(z) \, dz.
	\]
	Clearly, $\tilde{\Phi}_n$ is a completely positive map and let
	\[
	C_n=\tilde{\Phi}_n(I)= \int_{\mathbb{T}} |q_n(z)|^2 P(z)^* P(z) \, dz.
	\]
	Let $G(z) = P(z)^* f(z) P(z)$. Since $P(z)$ is a polynomial and $f(z)$ is continuous, $G \in C(\mathbb{T}, M_m(\mathbb{C}))$. The entry-wise convergence give us that 
	\[
	\lim_{n \to \infty} \int_{\mathbb{T}} G_{j,k}(z) d\mu_n(z) = \sum_{i=1}^N \lambda_i G_{j,k}(x_i).
	\]
	Because $M_m(\mathbb{C})$ is a finite-dimensional space, entry-wise convergence is equivalent to operator norm convergence. Thus,
	\[
	\lim_{n \to \infty} \left\| \tilde{\Phi}_n(f) - \sum_{i=1}^N \lambda_i P(x_i)^* f(x_i) P(x_i) \right\|_{M_m(\mathbb{C})} = \;\;0.
	\]
	Substituting back the values of $P(x_i) = A_i$ gives
	\[
	\lim_{n \to \infty} \tilde{\Phi}_n(f) = \sum_{i=1}^N \lambda_i A_i^* f(x_i) A_i = \Phi'(f).
	\]
	This pointwise norm convergence for every $f$ implies that $\tilde{\Phi}_n \to \Phi'$ in the BW topology, meaning $\rho \left( \tilde{\Phi}_n, \Phi'\right) \to 0$.
	And we also get that $C_n$ converges to $I_m$, which implies that $C_n$ is invertible for all $n$ sufficiently large. Now define $Q_n(z)=\tilde{Q}_n(z)(C_n)^{(-1/2)}$ and induced map $\Phi_n$ to be 
	\[
	\Phi_n(f) =  \int_\T Q_n(z)^*f(z)Q_n(z)\;dz = (C_n)^{(-1/2)}\; \tilde{\Phi}_n(f) (C_n)^{(-1/2)}.
	\]
	This gives that $\Phi_n$ is a \ucp map. Hence, we can choose $n$ large enough so that $\rho(\Phi_n, \Phi') < \frac{\epsilon}{4}$. Applying the triangle inequality across all steps and using previous lemmas
	\[
	\rho(\Phi, \Phi_n) \le \rho(\Phi, \Phi_0) + \rho(\Phi_0, \Phi') + \rho(\Phi', \Phi_n) < \frac{\epsilon}{2} + \frac{\epsilon}{4} + \frac{\epsilon}{4} = \epsilon.
	\]
	Setting $Q_V(z) = Q_n(z)$ and $\Phi_{Q_{V}} = \Phi_n$ completes the entire proof.
	
	\subsection{Proof of Theorem ~\ref{thm:hauscong}}
There are many equivalent definitions of Hausdorff distance, we will use the following from ~\cite{Burago}. It is 
	\[ d_H(\Bmn, \Ym) = \max \left\{ \sup_{\Phi_{Q_V} \in \Bmn} \rho(\Phi_{Q_V}, \Ym), \; \sup_{\Phi \in \Ym} \rho(\Phi, \Bmn) \right\} 
	\]
	Since $\Bmn$ is a subset of $\Ym$ we get that 
	\begin{equation}
		\sup_{\Phi_{Q_V} \in \Bmn} \rho(\Phi_{Q_V}, \Ym) = 0 \quad \forall n \in \mathbb{N}.
	\end{equation}
	Now to evaluate 
	\[
	\sup_{\Phi \in \Ym} \rho(\Phi, \Bmn) = \sup_{\Phi \in \Ym} \left( \inf_{\Psi \in \Bmn} \rho(\Phi, \Psi) \right).
	\]
	Let $\varepsilon > 0$ be given. For each $\Phi \in \Ym$, construct an open ball centered at $\Phi$ with radius $\frac{\epsilon}{2}$
	\[
	U_\Phi = \left\{ \Psi \in \Ym : \rho(\Phi, \Psi) < \frac{\epsilon}{2} \right\}.
	\]
	The collection of open sets $\{U_\Phi\}_{\Phi \in \Ym}$ forms an open cover of the compact set $\Ym$. Thus, there exists a finite collection of points $\{\Phi_1, \Phi_2, \dots, \Phi_k\} \subset \Ym$ such that
	\[
	\Ym \subseteq \bigcup_{i=1}^k U_{\Phi_i}.
	\]
	For each $\Phi_i$, using Theorem ~\ref{thm:existenceQ_V}, there exists an integer $N_i \in \mathbb N$ such that for every $n \ge N_i$, there exists a matrix-valued polynomial $Q_{i,n}$ of degree $n$ with roots on $\mathbb{T}$ and the induced map $\Phi_{Q_{i,n}} \in \Bmn$ satisfies
	\[
	\rho(\Phi_i, \Phi_{Q_{i,n}}) < \frac{\epsilon}{2}.
	\]
	Define $N=  \max \left\{ N_1, N_2, \ldots, N_k \right\}$. Now, let $\Phi$ be an element in $\Ym$ and consider an integer $n \ge N$. Because the finite balls cover the entire space, $\Phi$ must belong $U_{\Phi_i}$ for some $i$, meaning $\rho(\Phi, \Phi_i) < \frac{\epsilon}{2}$.
	Hence,
	\[
	\rho(\Phi, \Bmn) = \inf_{\Psi \in \Bmn} \rho(\Phi, \Psi) \le \rho(\Phi, \Phi_{Q_{i,n}}) \le \rho(\Phi, \Phi_i) + \rho(\Phi_i, \Phi_{Q_{i,n}}).
	\]
	\[
	\rho(\Phi, \Bmn) < \frac{\epsilon}{2} + \frac{\epsilon}{2} = \epsilon \quad \forall n \ge N.
	\]
	Since this strict inequality holds uniformly for every $\Phi \in \Ym$, it bounds the supremum over the entire space
	\[
	\sup_{\Phi \in \Ym} \rho(\Phi, \Bmn) \le \epsilon \quad \forall n \ge N.
	\]
	Because $\epsilon > 0$ was chosen arbitrarily, we get that
	\begin{equation}
		\lim_{n \to \infty} \left( \sup_{\Phi \in \Ym} \rho(\Phi, \Bmn) \right) = 0.
	\end{equation}
	Now, combining both the results, we get that 
	\[
	\lim_{n \to \infty} d_H(\Bmn, \Ym) = \lim_{n \to \infty} \max \left\{ \sup_{\Phi_{Q_V} \in \Bmn} \rho(\Phi_{Q_V}, \Ym), \; \sup_{\Phi \in \Ym} \rho(\Phi, \Bmn) \right\} =0.
	\]

\section{Epilogue: The Gromov-Hausdorff Convergence}
We proved the Hausdorff convergence in detail because the method was new. An expected consequence of the Hausdorff convergence follows. The proof is a minor modifications of ~\cite{Hek}.

\begin{defn}[Gromov-Hausdorff convergence] 
	The Gromov-Hausdorff distance between two metric spaces $(X,d_1)$ and $(Y,d_2)$, to be denoted as $d_{GH}$, is defined to be the infimum of all $r>0$ such that there exists a metric space $Z$ with subsets $X_1, Y_1 \subseteq Z$ isometric to $X$ and $Y$, respectively, with $d_H(X_1, Y_1)<r$, where $d_H(X_1,Y_1)$ is the Hausdorff distance between $X_1$ and $Y_1$.
\end{defn}

We make $\mathcal{UCP}_{\mathbb C^m}(\cT_{n,m})$ (this notation introduced in the beginning of Section ~\ref{PureUCP} stands for all \ucp maps from $\cT_{n,m}$ into $M_m(\mathbb C)$) a metric space by defining $\rho_n : \mathcal{UCP}_{\mathbb C^m}(\cT_{n,m}) \times \mathcal{UCP}_{\mathbb C^m}(\cT_{n,m}) \to [0, \infty)$  as
\[ \rho_n(\phi, \psi) =  \sup \left\{ \|\phi(T) - \psi(T)\|_{M_m(\mathbb{C})} : T \in \cT_{n,m}, \, \|T\| \le 1\; \text{and}\; \|[D_n, T]\| \le 1 \right\},
\]
where $D_n \in M_{mn}(\C)$ is the {\em Dirac operator} $\text{diag}(1,\ldots,n) \otimes I_m$. This metric induces the bounded weak (BW) topology on $\mathcal{UCP}_{\mathbb C^m}(\cT_{n,m})$.

Let $\mathcal{X}_{n,m}$ denote the space of all pure \ucp maps in $\mathcal{UCP}_{\mathbb C^m}(\cT_{n,m})$. Define the map $R_n :C(\T,M_m(\C)) \to \cT_{n,m}$ by $f\mapsto P_nfP_n$, where $P_n$ is the orthogonal projection as defined in Section ~\ref{sec:intro}. It can be easily checked that $\mathcal{X}_{n,m} \circ R_n := \{ \Phi \circ R_n : \Phi \in \mathcal{X}_{n,m} \} = \Bmn$. 
%Note that the space $\mathcal{X}_{n,m}$ need not be a compact space with respect to the restricted metric $\rho_n$. 
A combination of  Theorem ~\ref{thm:hauscong} and the techniques developed in ~\cite{Walter} imply the following.

\begin{theorem}
    $\mathcal{X}_{n,m}$ convereges to $\Ym$ in the Gromov-Hausdorff distance as $n \to \infty.$
\end{theorem}

\end{document}